\documentclass[12pt]{amsart}

\usepackage{graphicx}

\usepackage[all,cmtip]{xy}

\usepackage{marvosym}

\usepackage{answers}
\Newassociation{solution}{Solution}{ans}
\newtheorem{exercise}{Exercise}

\theoremstyle{definition}
\newtheorem{definition}{Definition}
\newtheorem{example}{Example}

\theoremstyle{plain}

\newtheorem{proposition}{Proposition}

\author{J. R. Arteaga, M. Malakhaltsev}
\title[$G$-structures and differential equations]{Ideas of E.~Cartan and S.~Lie in modern geometry: $G$-structures and differential equations. 
Lecture 2}
\address{}
\email{}

\begin{document}
\maketitle
\Opensolutionfile{ans}[answers_lecture_2]
\medskip
\fbox{\fbox{\parbox{5.5in}{
\textbf{Problem:}\\
How to find invariants of a geometric structure?
}}}
\vspace{1cm}

In the first lecture we have seen an example of how the invariants help us in the problem of classification of differential equations.
In general, the invariants play a very important role in the theory of geometrical structures and their applications. 
But how one can find the invariants of a geometrical structure?
 
To do it, we first consider an \emph{adapted frame field}, a frame field which is adapted  
to a geometrical object.
Then we derivate the vector fields of the frame and expand the derivatives in terms of the same frame. 
If the adapted frame is uniquely associated to a geometric object, then, any motion maps the adapted frame of the object to the adapted frame of its image. 
Therefore, the coefficients of the expansion will be the same, so they are invariants of the geometric object. 

The idea to use the coefficients of the derivation equations of an adapted frame in order to construct invariants of geometric objects (such as surfaces, Riemannian manifolds and submanifold, distributions, etc.) is due to \'E.\,Cartan, and is widely used in the modern differential geometry. 

In this lecture we will give examples of how this idea can be applied to the specific examples, and give introduction to the theory of $G$-structures which provides the background for development of Cartan's ideas in the framework of modern geometry.  

\section{Motivation for the theory of $G$-structures. Moving frames}

\subsection{Example of adapted frame field: moving frame of a plane curve}
Let $\gamma$ be a regular curve on the oriented plane $\mathbb{R}^2$ given by parametrization $\vec{r}=\vec{r}(s)$.
Then at each point $\vec{r}(s)$ we take the unit tangent vector field $\vec{T}(s)$ and the vector $\vec{N}(s)$ obtained by counterclockwise rotation on angle $\pi/2$. 
\begin{figure}[h]
\begin{center}
\includegraphics[scale=0.1]{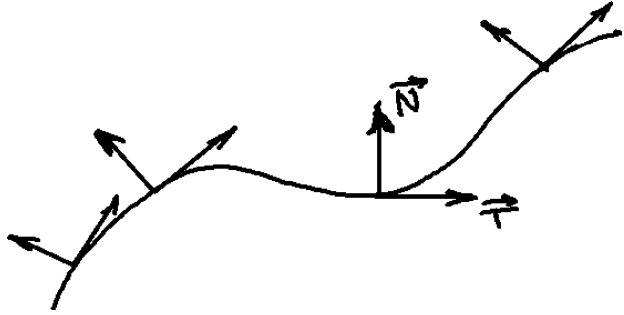}
\end{center}
\caption{Frenet frame}
\label{fig:Frenet_frame_2}
\end{figure}
Thus we obtain a frame at each point, so a \emph{frame field} $\{ \vec{T}(s),\vec{N}(s)\}$, called the \emph{Fren\'et frame field}, which is \emph{unique} if the curve is oriented. 

The \emph{derivation equations} of the Fren\'et frame field are the famous \emph{Fren\'et} equations: 
\begin{eqnarray}
\vec{T}'(s) &=& k(s) \vec{N}(s)
\\
\vec{N}'(s) &=& - k(s)\vec{T} 
\label{eq:Frenet_equations_2}
\end{eqnarray}
The coefficient function $k(s)$ is called the \emph{curvature function} of the curve and is an invariant of the curve. 
Indeed if we have a motion (an isometry) $A : \mathbb{R}^2 \to \mathbb{R}^2$ of the Euclidean plane $\mathbb{R}^2$, and $\widetilde{\gamma} = A(\gamma)$, then $\{A(\vec{T}),A(\vec{N})\}$ 
is the Fren\'et frame for $\widetilde{\gamma}$. Therefore, if $\widetilde{k}(s)$ is the curvature of $\widetilde{\gamma}(s)$, then $\widetilde{k}(s) = k(s)$.   
\begin{figure}[h]
\begin{center}
\includegraphics[scale=0.5]{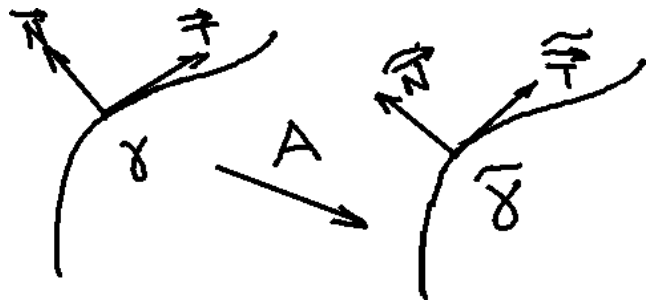}
\end{center}
\caption{Isometry action on the Fren\'et frame}
\label{fig:Frenet_frame_isometry_action_2}
\end{figure}

\begin{exercise}
a) Find the curvature of straight line and a circle of radius $r$. 
Does an isometry exists which maps a circle of radius $3$ to a circle of radius $2$?

b) What is the adapted frame for a curve in the three-dimensional space $\mathbb{R}^3$? What are the invariants of this curve? 
\begin{solution}
a) The curvature of a straight line is $0$. The curvature of a circle of radius $r$ is $1/r$. 
The isometry does not exist because $1/r \ne 0$.
b) The Fren\'et frame. The curvature and the torsion.
\end{solution}
\end{exercise}

\subsection{Example of adapted frame: surfaces in  $\mathbf{R}^{3}$. Invariants of surfaces}

If we have a surface $\Sigma \subset \mathbb{R}^3$ given by a parametric equation $\vec{r} = \vec{r}(u^1,u^2)$, then we can construct an adapted frame field in the following way. 
For the first two vectors of the frame we take 
\begin{equation}
\vec{e}_1 = \partial_1\vec{r}, \vec{e}_2 = \partial_2\vec{r},
\label{eq:adapted_frame_surface_12}
\end{equation}
and the third one is the unit normal vector 
\begin{equation}
\vec{n} = \frac{1}{\| \vec{e}_1 \times \vec{e}_2 \|} \vec{e}_1 \times \vec{e}_2.  
\label{eq:adapted_frame_surface_3}
\end{equation}
Then the \emph{derivation equations} of this frame are
\begin{equation}
\begin{split}
&\partial_i \vec{e}_j = \Gamma^k_{ij} \vec{e}_k + h_{ij} \vec{n},
\\
& \partial_i \vec{n} = - h^s_i \vec{e}_s.
\end{split}
\label{eq:derivation_equations_surface}
\end{equation}
The coefficients of the derivation equations  play important role in the surface theory ($\Gamma^k_{ij}$ are \emph{connection coefficients}, $h_{ij}$ is the \emph{second fundamental form}, $h^i_j$ is the \emph{shape operator}), 
however the adapted frame \emph{is not unique} because we can change parametrization, therefore the \emph{coefficients themselves are not the invariants of the surface}. 

Let us look what happens if we change the parametrization, and so change the adapted frame.
If $u^{k'} = u^{k'}(u^k)$, then
\begin{equation}
\begin{split}
& \vec{e}_{1} = \frac{\partial u^{1'}}{\partial u^1} \vec{e}_{1'} + \frac{\partial u^{2'}}{\partial u^1} \vec{e}_{2'},
\\  
& \vec{e}_{2} = \frac{\partial u^{1'}}{\partial u^2} \vec{e}_{1'} + \frac{\partial u^{2'}}{\partial u^2} \vec{e}_{2'},
\\
&\vec{e}_{3} = \vec{e}_{3'} 
\end{split}
\label{eq:change_of_frame_surface}
\end{equation}
Therefore, the transformation matrices form a group
\begin{equation}
G = 
\left\{
A = \left[
\begin{array}{ccc}
A^{1'}_{1} & A^{1'}_{2} & 0
\\
A^{2'}_{1} & A^{2'}_{2} & 0
\\
0 & 0 & 1
\end{array}
\right]
\right\}
\label{transformation_group_surface}
\end{equation}
Under the transformations $A \in G$ the matrix $h = [ h^k_i ]$ change: 
\begin{equation}
h' = A^{-1} h A,
\label{eq:}
\end{equation}
but 
\begin{equation}
K = \det  h , \text{ and } H = \text{tr}\, h  
\label{eq:gauss_mean_curvature}
\end{equation}
are invariant under the action of the group $G$,
so \emph{$K$ and $H$ are invariants of the surface} (in fact, these are the famous  the \emph{gauss curvature} and the \emph{mean curvature}, respectively).
\begin{exercise}
Does there exist a motion of $\mathbb{R}^3$ which maps an ellipsoid to a hyperboloid?
\begin{solution}
No, because the gaussian curvature of an ellipsoid is positive and the gaussian curvature of a hyperboloid is negative.
\end{solution}
\end{exercise}

\section{G-structures}
The examples above explain the idea to construct invariants from the coefficients of derivation equations.
Technically this idea is realized nowadays using the concept of $G$-structure. 

\subsection{Examples of adapted coframes} 
\label{subsec:examples_of_adapted_coframes}

In the classical differential geometry \emph{adapted coframe} means a coframe field which is, in some sense, adapted to a given geometrical structure. 
Let us consider several examples of classical geometric structures.

\begin{example}[Riemannian metric]
\label{ex:riemannian_metric}
If $g$ is a Riemannian metric given on a $n$-dimensional manifold $M$, then the adapted coframes are the orthonormal coframes of $g$, that is coframes $e^i$ such that $g = e^1\otimes e^1 + \dots e^n\otimes e^n$. 
The matrix of $g$ with respect to an orthonormal coframe is the identity matrix.  
\end{example}

\begin{example}[Almost complex structure]
\label{ex:almost_complex_structure}
An almost complex structure on a $2m$-dimensional manifold $M$ is a linear operator field $J$ such that $J^2=-I$. 
The adapted coframes are the coframes with respect to which the matrix of $J$ is 
\begin{equation}
\left[
\begin{array}{rr}
0 & -I
\\
I & 0
\end{array}
\right],
\label{eq:matrix_of_almost_complex_structure}
\end{equation} 
where $I$ is the identity $m\times m$-matrix.
\end{example}

\begin{example}[Presymplectic form]
\label{ex:presymplectic_form}
A presymplectic form is a $2$-form $\omega$ on a $2m$-dimensional manifold which is nondegenerate at all points. The adapted coframes are the coframes $\left\{ e^k \right\}_{k=1,2m}$ such that 
\begin{equation}
\omega = e^1 \wedge e^2 + \dots + e^{2m-1} \wedge e^{2m}
\label{eq:canonical_frame_of_presimplectic_form}
\end{equation} 
\end{example}

\begin{example}[Distribution]
\label{ex:distribution}
If $\Delta$ is a $k$-dimensional distribution on an $n$-dimensional manifold $M$, then the adapted coframes  $\left\{ e^1, \dots e^n \right\}$ are such that $\Delta$ is given by $n-k$ linear equations 
\begin{equation}
e^{k+1}=0, e^{k+2}=0, \dots e^n = 0. 
\label{eq:linear_equations_of_distribution}
\end{equation}

Let $x$ be a point of $M$. Then any coframe $e = \left\{ e^1, \dots, e^k \right\}$ of $T_x M$ determines an isomorphism $\varphi_e$  between the Grassmannian $G_k(T_x M)$ (the space of all $k$-dimensional subspaces of $T_x M$) and the Grassmannian $G_k(\mathbb{R}^n)$. 
So, an adapted coframe for a distribution $\Delta$ is a frame $e$ such that $\varphi_e$ maps $\Delta(x)$ to $\mathbb{R}^k \subset \mathbb{R}^n$, where $\mathbb{R}^k = \left\{ (0, \dots 0, x^{k+1}, \dots x^n) \in \mathbb{R}^n \right\}$. 
\end{example}

\begin{exercise}
What is the set of coframes adapted to a nonvanishing vector field $v$ on the plane $\mathbb{R}^2$?
\begin{solution}
For example, the set of coframes $\left\{ e^1,e^2 \right\}$ such that $e^1(v)=1$, $e^2(v)=0$. 
\end{solution}
\end{exercise}

In order to consider the previous examples from a general point of view we need a portion of the frame bundle theory. 
We will not go in details here, so we refer the reader to \cite{KN}, \cite{Montgomery} where he can find the detailed exposition of the theory. 
However, for the reader who is familiar with the theory of bundles, we give some additional information in the text starting with the symbol \Info\ and ending with  \Squarepipe.

\subsection{Linear coframe bundle}
\label{subsec:linear_coframe_bundle}
As we would like to consider all the adapted frames at all points of a manifold it is natural to consider the set of all frames at all points. 
However, for our purposes, it is convenient to consider the set of all coframes of the tangent spaces at all the points of a manifold.  

Recall that, for a vector space $V$, a linear functional $f : V \to \mathbb{R}$ is called a covector, and the set of all covectors endowed by the standard operations of addition of real valued functions and multiplication of a real valued function by a number, is a vector space $V^*$, called the \emph{covector space dual to $V$}. 
The dimension of $V^*$ is equal to the dimension of $V$. 
A \emph{coframe of $V$} is a frame of $V^*$.     

If $\left\{ e_1, \cdots, e_m \right\}$ is a frame of $V$, there exists a unique coframe $\left\{ e^1, \cdots, e^m \right\}$ such that $e^k(e_i)=\delta^k_i$. 
Moreover, if $v = v^1 e_1 + \cdots + v^m e_m$, then $e^i(v) = v^i$, $i=\overline{1,m}$.

For a manifold $M$, and $x \in M$,  denote by $B(T_x M)$ the set of all coframes of $T_x M$. 
The group $GL(m)$ acts on $B(T_x M)$ from the right: 
\begin{equation}
\forall e = (e^1,\dots, e^m)^t \in B(T_x M),\ A \in GL(m), \quad (e,A) \mapsto A^{-1}e, 
\label{eq:GL_action_on_coframes}
\end{equation}
It is clear that this action is free and transitive.

Now we set 
\begin{equation}
B(M) = \mathop\sqcup_{x \in M} B(T_x M), 
\label{eq:set_of_all_coframes}
\end{equation}
and we can define the surjective map, called the \emph{projection}, 
\begin{equation}
\pi : B(M) \to M, \quad \left\{ e^1, \cdots, e^m \right\} \in B(T_x M) \to x \in M.
\label{eq:projection_coframe_bundle}
\end{equation}
The subsets $\pi^{-1} = B(T_x M) \subset B(M)$ are called the \emph{fibers of} $\pi$.
The action \eqref{eq:GL_action_on_coframes} induces a free action of the group $GL(m)$ on the set $B(M)$ which is transitive at the fibers. 

Now take an atlas $(U_\alpha,x^i_\alpha)$ of the manifold $M$. Then we have the bijections 
\begin{equation}
\psi_\alpha : \pi^{-1}(U_\alpha) \to U_\alpha \times GL(n), 
\quad
\psi(A^i_j dx^j_\alpha|_p) = (p, ||A^i_j||^{-1}),   
\label{eq:coframe_bundle_trivializations}
\end{equation} 
and
\begin{equation}
\psi_\beta \psi_\alpha^{-1} : (U_\alpha \cap U_\beta) \times GL(n) \to (U_\alpha \cap U_\beta) \times GL(n),
\ (p, A) \to (p, \frac{Dx_\beta}{Dx_\alpha} A),
 \label{eq:coframe_bundle_transition_maps}
\end{equation}
where $\frac{Dx_\beta}{Dx_\alpha}$ is the Jacobi matrix of the coordinate change. 
Using this bijections, one can introduce the manifold structure on the set $B(M)$ so that $\pi : B(M) \to M$ is a smooth map, and the action of $G$ on $B(M)$ is also smooth.

The triple $(B(M), \pi, M)$ is called \emph{coframe bundle of a manifold $M$}.
The coframe bundle is an example of a principal fiber bundle.

\subsection{Principal fiber bundle}
Let $P$ be a manifold endowed by a free right action of a Lie group $G$. Assume that the $M = P/G$ is a smooth manifold, and the natural projection 
\begin{equation}
\pi : P \to M = P/G, \quad p \to [p]
\label{eq:principal_bundle}
\end{equation}
is a smooth map. 
In addition assume that the map $\pi$ is \emph{locally trivial}. 
This means that there exists a open covering $\left\{ U_\alpha \right\}$ of $M$ and diffeomorphisms 
$\varphi_\alpha : \pi^{-1}(U_\alpha) \to U_\alpha \times G$ with the properties:
\noindent
1) if $\varphi_\alpha(p) =  (x,h)$, then $\varphi(pg) = (x,hg)$ ($\varphi_\alpha$ is \emph{$G$-equivariant});
\noindent
2) the following diagrams are commutative:
\begin{equation}
\xymatrix{
\pi^{-1}(U_\alpha)  \ar[rd]_{\pi}  \ar[rr]_{\varphi_\alpha} & &  U_\alpha \times G \ar[ld]
\\
& U_\alpha &
}
\label{eq:trivializing_covering}
\end{equation}
Then $\pi : P \to M$ is called a \emph{$G$-principal bundle}.

One can see, that with above assumptions, the map  $\varphi_\beta \varphi_\alpha^{-1}$ is  
\begin{equation}
\varphi_\beta \varphi_\alpha^{-1} : U_\alpha \cap U_\beta \times G \to U_\alpha \cap U_\beta \times G,
\quad (x,g) \to (x, g_{\beta\alpha}g).
\end{equation}
The maps 
\begin{equation}
g_{\beta\alpha} : U_\alpha \cap U_\beta \to G
\end{equation}
are called the \emph{gluing maps}, because $P$ can be viewed as a space glued from \emph{trivial bundles} $U_\alpha \times G$ via the maps $g_{\beta\alpha}$.

Thus the coframe bundle is a  $GL(m)$-principal bundle. 

The sets of adapted frames given in examples in \ref{subsec:adapted_coframes} are principal subbundles of the coframe bundle. 
\begin{example}
The $3$-web considered in Lecture 1 determines a $G$-structure on $\mathbb{R}^2$, where $G$ is the group of scalar matrices. 
\end{example}

\begin{example}[Riemannian metric]
\label{ex:riemannian_metric_2}
The adapted coframes of a  Riemannian metric $g$ given on a manifold $M$ form a $SO(m)$-principal bundle $SO(M,g) \to M$.
\end{example}

\begin{example}[Almost complex structure]
\label{ex:almost_complex_structure_2}
The adapted coframes of an almost complex structure on a $2m$-dimensional manifold $M$ form a $GL(m,\mathbb{C})$-principal bundle $L(M,J) \to M$. 
\end{example}

\begin{example}[Presymplectic form]
\label{ex:presymplectic_form_2}
The adapted coframes of  a presymplectic form $\omega$ on a $2m$-dimensional manifold form a $Sp(m)$-principal bundle 
$Sp(M,\omega) \to M$. 
\end{example}

\begin{example}[Distribution]
\label{ex:distribution_2}
The adapted coframes of a $k$-dimensional distribution  $\Delta$ on an $n$-dimensional manifold $M$ form a $GL(n,k)$-principal bundle, where  
\begin{equation}
GL(n,k) = \left\{ 
\left[
\begin{array}{cc}
A & B
\\
0 & C
\end{array}
\right]
\mid A \in GL(k), C \in GL(n-k)
 \right\}.
\label{eq:group_distribution}
\end{equation}
\end{example}

\begin{exercise}
a) Prove that for the torus $\mathbb{T}^2$ we have $B(\mathbb{T}^2) \cong \mathbb{T}^2 \times GL(2)$.

b) Is it true that for the $2$-dimensional sphere $\mathbb{S}^2$ we have  $B(\mathbb{S}^2) \cong \mathbb{S}^2 \times GL(2)$?
\begin{solution}
b) No.
\end{solution}
\end{exercise}
\begin{definition}
A \emph{$G$-structure } is a $G$-principal subbundle $P(M;G) \to M$ of the coframe bundle $B(M)$. 
\end{definition}

\subsection{Associated bundles}
\label{subsec:associated_bundles}
Let $\pi : P \to M$ be a $G$-principal fiber bundle, and $\rho : G \times Y \to Y$ be a left action of $G$ on a manifold $Y$. 
Let $\{U_\alpha, \psi_\alpha : \pi^{-1}(U_\alpha) \to U_\alpha \times G\}$ be an atlas of the bundle $\pi: P \to M$ with the gluing maps $\psi_{\beta\alpha} : U_\alpha \cap U_\beta \to G$. 
Then the maps 
\begin{equation}
\hat\psi = U_\alpha \cap U_\beta \to Diff(Y), \quad \hat\psi(x) = \rho(\psi(x)), x \in U_\alpha \cap U_\beta,    
\label{eq:gluing_maps_of_associated_bundle}
\end{equation}
satisfy the \emph{cocycle conditions}: 
\begin{equation}
\begin{array}{l}
\textrm{1)}
\hat{\psi}_{\alpha\alpha}(x) = 1_Y, \forall x \in U_\alpha,
\\
\textrm{2)}
\hat{\psi}_{\alpha\beta}(x) \circ \hat{\psi}_{\beta\gamma}(x) \circ \hat{\psi}_{\gamma\alpha}(x)= 1_Y, \quad \forall x \in U_\alpha \cap U_\beta \cap U_\gamma.
\end{array}
\label{eq:cocycle_condition}
\end{equation}
therefore they are gluing maps of a bundle, called the \emph{bundle associated with the $G$-principal bundle $\pi: P \to M$ with respect to the action $\rho$}.

\begin{example}[Tangent bundle as associated bundle]
\label{ex:tangent_bundle_as_associated_bundle}
Consider the coframe bundle $B(M) \to M$ over an $m$-dimensional manifold $M$. 
In this case the group $G = GL(m)$. Consider the action of $GL(m)$ on the vector space $\mathbb{R}^m$:
\begin{equation}
\rho: GL(m) \times \mathbb{R}^m \to \mathbb{R}^m, \quad \rho(A,v) =  Av
\label{eq:action_of_GL(m)_on_R_m}
\end{equation}
Then the associated bundle $E$ is the bundle with gluing functions 
\begin{equation}
v_\beta = \rho(\frac{Dx_\beta}{Dx_\alpha}) v_\alpha, \text{ which translates to }
v_\beta^k = \frac{\partial x_\beta^k}{\partial x_\alpha^i} v_\alpha^i  
\label{eq:gluing_functions_of_tangent_bundle}
\end{equation} 
As this is exactly the transformation law for the vector components under the coordinate change, we see that the elements of $E$ are the tangent vectors, and $E = TM$ is the tangent bundle of $M$. 
\end{example}

\par\noindent
{\large \Info} 
\par\noindent
The associated bundle can be also described as follows. 
On the manifold $P \times Y$ we introduce the equivalence relation $(p,y) \sim (p g, \rho(g^{-1} y)$, for $p \in P$, $y \in Y$, and $g \in G$. By $\pi_{P \times Y}$ we denote the canonical projection of $P \times Y$ onto the set $E = P \times Y/\sim$ of equivalence classes.

The set $E$ can be endowed with a manifold structure, so that  $\pi_E : E \to M$,  $\pi_{P \times Y}(p,y) = [p,y] \mapsto \pi(p)$ is a fiber bundle with typical fiber $Y$, and this bundle is exactly the associated bundle. 
\begin{example}
Consider again Example~\ref{ex:tangent_bundle_as_associated_bundle}
Then $P = B(M)$, $Y = \mathbb{R}^m$, the bundle $E$ is the tangent bundle $TM$, and 
the map $\pi_{B(M) \times \mathbb{R}^m}$ is
\begin{equation}
\pi_{B(M) \times \mathbb{R}^m} : B(M) \times \mathbb{R}^m \to TM, 
\quad 
\left( \{e^1,\dots,e^m\} , (v^1,\dots,v^m)^t)\right) \to v^1 e_1 + \dots + v^m e_m,
\end{equation}  
where $\{e_1,\dots,e_m\}$ is the frame dual to the coframe $\{e^1,\dots,e^m\}$.
\end{example}
\begin{proposition}
\label{lem:coordinate_diffeomorphism_p_E}
To each $p \in P$ we can associate a unique diffeomorphism $p^E : E_{\pi(p)} \to Y$ such that 

1) $p^E([p, y]) = y$. 

2) $(p g)^E = g^{-1} p^E$.
\end{proposition}
\begin{proof}
1) follows from the construction of the associated bundle.

2) We have $(pg)^E([p g, g^{-1}y]) = g^{-1}y$. Therefore, $(pg)^E([p , y]) = g^{-1}y = g^{-1}p^E([p,y])$.
\end{proof}
The diffeomorphism $p^E$ will be called the \emph{coordinate diffeomorphism}, 
\begin{example}
Under the setting of the previous example, for a coframe 
$p = \left\{ e^1,\dots,e^m \right\}$, the map $p^{TM} : T_x M \to \mathbb{R}^m$ is the map 
\begin{equation}
p^{TM}(v) = (v^1=e^1(v),\dots,v^m=e^m(v)).
\end{equation}  
that is, $p^{TM}$ gives the coordinates of a vector $v \in T_{\pi(p)}M$ with respect to the coframe $p$.
\end{example}
So each $p \in P$ determines an isomorphism, called the \emph{coordinate isomorphism}, of a fibre of associated bundle over $\pi(p)$ to the typical fibre $Y$.

The \emph{local trivializations} of the associated bundle can be described in terms of the coordinate diffeomorphisms as follows.
Let us take a covering $\{U_\alpha\}$ of $M$ such that on each $U_\alpha$ we have a section $p_\alpha : U_\alpha \to P$  of the bundle $\pi : P \to M$. 
These sections determine trivializations 
\begin{equation}
\psi_\alpha : U_\alpha \times G \to \pi^{-1}(U_\alpha), \quad \psi(x,g) = p_\alpha(x) g 
\end{equation}
of the bundle $P$. 
Then the sections $p_\alpha$ also determine trivializations 
$\psi^E_\alpha : \pi_E^{-1}(U_\alpha) \to U_\alpha \times Y$ 
of the bundle $E$:
\begin{align}
&\psi^E_\alpha(e) = (\pi_E(e), p_\alpha(\pi_E(e)))^E (e)) 
\\
&(\psi^E_\alpha)^{-1}(x,y) = (p_\alpha(x)^E)^{-1}(y).
\end{align}

\begin{proposition}
\label{lem:coordinate_diffeomorphism_wrt_trivializations}
In terms of the trivializations $\psi_\alpha$ and $\psi^E_\alpha$ 
the map $p^E : E_{\pi(p)} \to Y$ is written as follows: 
if $\pi(p)=x$, and $p = \psi_\alpha(x,g) = p_\alpha(x) g$,  then 
\begin{equation}
p^E(x,y) = g^{-1} y.
\label{eq:coordinate_diffeomorphism_wrt_trivializations}
\end{equation}
\end{proposition}
\begin{proof}
In terms of the trivializations $\psi^E_\alpha$ the map $p^E : E_{\pi(p)} \to Y$ is written as follows: if $\pi(p)=x$, 
and $p = p_\alpha(x) g$,  then 
\begin{multline}
(x,y) \to p^E (\psi^E_\alpha)^{-1}(x,y) = (p_\alpha(x) g)^E \circ (p_\alpha(x)^E)^{-1}(y) =
\\
= g^{-1} p_\alpha(x)^E \circ (p_\alpha(x)^E)^{-1}(y) = g^{-1} y.
\end{multline}
\end{proof}
\par\noindent
\textbf{\Squarepipe} 

\begin{proposition}
The sections $s : M \to E$ are in one-to-one correspondence with the $G$-equivariant maps from $P$ to $Y$, that is with the maps $f_s : P \to Y$ such that $f_s(pg)=\rho(g)^{-1} f_s(p)$. 

In terms of trivializations the map $f_s$ is given as follows:
\begin{equation}
(f_s)_\alpha(x,g) = g^{-1} y_\alpha(x).
\label{eq:local_representation_of_f_s_1}
\end{equation}
\end{proposition}
\begin{example}
Let $s : M \to TM$ be a section of the tangent bundle, that is a vector field. 
Then, the corresponding map $f_s$ is 
\begin{equation}
f_s : B(M) \to \mathbb{R}^m, \quad 
p=\left\{ e^1, \dots, e^m \right\} \in B(M) \to \left(e^1(v(\pi(p))),\dots, e^m(v(\pi(p)))\right).       
\end{equation}  
that is $f_s$ maps a coframe $p \in B(M)$ at a point $x \in M$ 
to the vector in $\mathbb{R}^m$ of coordinates of $v(x)$ with respect to this coframe. 
\end{example}

\par\noindent
{\large \Info}
Namely, given a section $s$, we define 
\begin{equation}
f_s(p) = p^E (s(\pi(p)).
\label{eq:definition_of_f_s}
\end{equation}
The $G$-equivariant property of $f_s$ follows immediately from Lemma~\ref{lem:coordinate_diffeomorphism_p_E}. 

Let $s : M \to E$ be a section, then define the local representation of the section $y_\alpha : U_\alpha \to Y$ as follows: 
\begin{equation}
s_\alpha = (\psi^E_\alpha) \circ s |_{U_\alpha}: U_\alpha \to U_\alpha \times Y, 
\quad s_\alpha(x) = (x, y_\alpha(x)). 
\end{equation}
Then, the corresponding map $f_s : P \to Y$, $p \mapsto p^E(s(\pi(p)))$, can be written in terms of the trivializations $\psi_\alpha$ and the local representations $y_\alpha$ of the section $s$: 
\begin{multline}
(f_s)_\alpha(x,g) = f_s \circ \psi_\alpha(x,g) = f_s (p_\alpha(x) g) =
\\
= (p_\alpha(x) g)^E (s(x)) = g^{-1} (p_\alpha(x))^E (s(x)) = g^{-1} y_\alpha(x). 
\end{multline}
Thus we have proved
\begin{proposition}
\label{lem:local_representation_of_f_s}
Assume that a section $s : M \to E$ has local representations $y_\alpha : U_\alpha \to Y$ with respect to the trivializations $\psi^E_\alpha$. Then the local representation $(f_s)_\alpha = f_s \circ \psi_\alpha : U_\alpha \times G \to Y$ of the map $f_s$ is 
\begin{equation}
(f_s)_\alpha(x,g) = g^{-1} y_\alpha(x).
\label{eq:local_representation_of_f_s}
\end{equation}
\end{proposition}
\par\noindent
{\large \Squarepipe}
 
\subsection{$G$-structure of adapted coframes constructed via a section of an associated bundle}
\label{subsec:adapted_coframes}
\begin{definition}
Let $P \to M$ be a $G$-principal bundle, and $G_1$ be a subgroup of $G$. 
If there exists a $G_1$-principal subbundle $P_1 \to M$ of $P$, then we say \emph{$P$ reduces to the bundle $P_1$}, or  \emph{$P$ reduced to $G_1$-principal bundle}.   
\end{definition}

Let $Y$ be a manifold endowed by a $G$-action $\rho$. 
Denote by $\mathcal{O}_G(Y)$ is the set of orbits of the action $\rho$. 
For each orbit $O \in \mathcal{O}(Y)$, denote by $\rho_O$ the transitive action of $G$ on $O$ induced by the action $\rho$. 

Let $P \to M$ be a $G$-principal bundle.
Denote by $E_O$ the bundle with typical fiber $O$ associated to the $G$-principal bundle $\pi: P \to M$ with respect to $\rho_O$.  Then the total space $E$ is split into the disjoint union of the total spaces $E_\mathcal{O}$. 

For each section $s : M \to E$, we set $M_O = \{x \in M \mid s(x) \in M_O\}$. 
In general, $M_O$ is not a submanifold. 
However, if $s$ is transversal to $E_O$, then $M_O$ is a submanifold of $M$. 
Note also that, if $f_s : P \to Y$ is the map corresponding to $s$, then $f^{-1}(O)$ consists of fibers of $P$, and $\pi(f^{-1}(O)) = M_O$. 

If we assume that $s : M \to E$ takes values only in one orbit $O \in \mathcal{O}(Y)$, and $G_0 \subset G$ is the isotropy subgroup of a point $y_0 \in O$, then $f^{-1}(y_0)$ is a $G_1$-subbundle of $P$, so $P$ reduces to a $G_1$-principal bundle. 
If we take another point $y_1 \in Y$, and $G_1$ is the isotropy subgroup of $y_1$, then $f^{-1}(y_1)$ is an $G_1$-subbundle of $P$, which is isomorphic to $f^{-1}(y_0)$.     

Let $\pi : B(M) \to M$ be the coframe bundle of a manifold $M$ and $Y$ be a manifold endowed with $GL(n)$-action $\rho : GL(n) \times Y \to Y$. 
We can construct the bundle $\pi_{E} : E \to M$ associated with $B(M)$ with respect to the action $\rho$. 
Then any section  $s : M \to E$ such that $f_s : B(M) \to Y$ takes values in one orbit $O$ of the action $\rho$, determines a reduction of $B(M)$ to a $G$-principal subbundle, where $G$ is the stationary subgroup of a point $y_0 \in O$. 
So any section of this type determines a reduction of the coframe bundle $B(M)$ to a $G_0$-subbundle $B_G(M) \to M$. 
The elements of $B_G(M)$ are called \emph{adapted coframes}, or \emph{coframes adapted to the section} $s$.   

\begin{example}[Vector field]
The action of $GL(m)$ on $\mathbb{R}^m$ has two orbits: $O = \left\{ v \ne 0 \right\}$ and $O_0 = \left\{ 0 \right\}$.
A vector field on a manifold $M$ is a section $s$ of the tangent bundle $TM$ which is the bundle associated to the coframe bundle $B(M)$. 
If the corresponding map $f_s : B(M) \to \mathbb{R}^m$ takes values in $O$, then vector field does not vanish.
\end{example}
\begin{example}[Riemannian metric]
\label{ex:riemannian_metric_1}
The bundle $T^2_0(M) \to M$ of symmetric tensors on a manifold $M$ is the bundle associated to $B(M)$ with respect to the standard action $\rho$ of $GL(m)$ on the space $T^2_0(\mathbb{R}^m)$ of tensors of type $(2,0)$ on $\mathbb{R}^m$: 
\begin{equation}
(\rho(A)t)(x,y) = t(A^{-1} x, A^{-1} y), \quad A \in GL(n), t \in S^2(\mathbb{R}^2). 
\label{eq:action_of_GL_on_tensors_(2,0)}
\end{equation}
Let $\{e^i\}_{i=\overline{1,m}}$ be the standard coframe of $\mathbb{R}^m$, and take the tensor  
\begin{equation}
g_0 = e^1 \otimes e^1 + e^2 \otimes e^2 + \dots + e^m \otimes e^m
\label{eq:standard_metric_on_Rn}
\end{equation}
in $T^2_0(\mathbb{R}^m)$. 
Let $O(g_0)$ be its orbit under the action $\rho$.  
A Riemannian metric $g$ given on a manifold $M$  is a section $g: M \to T^2_0(M)$ such that the corresponding map $f_g : B(M) \to T^2_0(\mathbb{R}^n)$ takes values in the orbit $O(g_0)$. 
The isotropy group of $g_0$ is the Lie group $O(n)$ of orthogonal matrices.
Therefore, a Riemannian metric $g$ determines a reduction of $B(M)$ to the $O(n)$-principal subbundle of orthonormal coframes, and the \emph{coframes adapted to $g$} are exactly the orthonormal coframes. 
\end{example}

\begin{example}[Almost complex structure]
\label{ex:almost_complex_structure_1}
The bundle $T^1_1(M) \to M$ of linear operators (tensors of type $(1,1)$) on a manifold $M$ is a bundle associated to $B(M)$ with respect to the standard action $\rho$ of $GL(n)$ on the space $T^1_1(\mathbb{R}^n) = (\mathbb{R}^n)^*\otimes \mathbb{R}$:
\begin{equation}
(\rho(A)t)(x,\xi) = t(A^{-1} x, A^*\xi), \quad A \in GL(n), t \in T^1_1(\mathbb{R}^2). 
\label{eq:action_of_GL_on_tensors_(1,1)}
\end{equation} 
Assume that the dimension of $M$ is even, that is $n=2m$, Let us denote by $I_k$ the indentity $k \times k$-matrix.
The orbit $O(J_0)$ of the matrix 
\begin{equation}
J_0 = 
\left[
\begin{array}{cc}
0 & -I_m
\\
I_m & 0
\end{array}
\right],
\label{eq:canonical_matrix_of_almost_complex_structure}
\end{equation}
consists of linear operators $J$ such that $J^2 = -I_{2m}$, and the isotropy group of $J_0$ is the matrix group 
\begin{equation}
GL(m,\mathbb{C}) \cong \{
C = \left[
\begin{array}{cc}
A & -B
\\
B & A
\end{array}
\right]    
\mid \det C \ne 0\}
\label{eq:GL(n,C)}
\end{equation}
\end{example}
An almost complex structure (see Example~\ref{ex:almost_complex_structure}) is the section $J : M \to T^1_1 M$ such that the corresponding map $f_J : B(M) \to T^1_1(\mathbb{R}^n)$ takes values in $O(J_0)$. 
Therefore an almost complex structure determines a reduction of $B(M)$ to a $GL(m,\mathbb{C})$-principal subbundle of $B(M)$ and the coframes adapted to $J$ are the elements of this subbundle. In fact, the adapted coframes are those with respect to which the matrix of $J$ is exactly $J_0$ \eqref{eq:canonical_matrix_of_almost_complex_structure}.  

\begin{example}[Presymplectic form]
\label{ex:presymplectic_form_1}
A presymplectic form field (see Example~\ref{ex:presymplectic_form}) is a section $\omega$ of the bundle $T^2_0(M)$ such that $f_\omega : B(M) \to T^2_0(\mathbb{R}^n)$ takes values in the orbit $O(\omega_0)$ with
\begin{equation}
\omega_0 = e^1 \wedge e^2 + \dots + e^{2m-1} \wedge e^{2m},
\label{eq:canonical_form_of_presimplectic_form}
\end{equation} 
where $\left\{ e^i \right\}_{i=\overline{1,n}}$ is the standard coframe of $\mathbb{R}^n$. 
The isotropy group of $\omega_0$ is the group of symplectic transformations
\begin{equation}
Sp(m) = \left\{ A \mid A^t \omega_0 A = \omega_0 \right\}.
\label{eq:group_of_symplectic_transformations}
\end{equation}  
Thus a presymplectic form field determines a reduction of $B(M)$ to a $Sp(m)$-principal subbundle of $B(M)$.  
The adapted coframes are the coframes $\left\{ e^k \right\}_{k=1,2m}$ such that with respect to these coframes the tensor $\omega(x)$ is given by~\eqref{eq:canonical_form_of_presimplectic_form} for any $x \in M$.
\end{example}

\begin{example}[Distribution]
\label{ex:distribution_1}
Denote by $G_k(\mathbb{R}^n)$ the Grassmann manifold of $k$-dimensional subspaces in $\mathbb{R}^n$, and by $G_k(M) \to M$ the Grassmann bundle of $k$-dimensional subspaces of the tangent spaces of a manifold $M$. 
This bundle is associated to $B(M)$ with respect to the action of $GL(n)$ on the Grassmann manifold $G_k(\mathbb{R}^n)$: 
\begin{equation}
\rho(A) \Delta = A(\Delta), \quad \Delta \in G_k(\mathbb{R}^n).
\label{eq:action_of_GL_on_Grassmannian}
\end{equation}
This action has the unique orbit and the isotropy group of the $k$-dimensional subspace 
\begin{equation}
\Delta_0 = \left\{ e^{k+1} = 0, \dots e^n = 0 \right\}
\label{eq:standard_k_plane_in_Rn}
\end{equation}
is the matrix group 
\begin{equation}
G = \left\{
\left[
\begin{array}{cc}
A & B
\\
0 & C
\end{array}
\right]
\mid 
A \in GL(k), B \in Mat(k,n-k), C \in GL(n-k)
\label{eq:group_for_grassmannian}
\right\}.
\end{equation}
A $k$-dimensional distribution on $M$ is a section $\Delta : M \to G_{k}(M)$ and $f_\Delta : B(M) \to G_k(\mathbb{R}^n)$ takes values in one orbit of $GL(n)$-action $\rho$ (because the orbit is unique).  
So any distribution $\Delta$ determines a reduction of $B(M)$ to a $G$-principal subbundle, and the adapted coframes are those with respect to which $\Delta(x)$ is given by equations \eqref{eq:linear_equations_of_distribution}.  
\end{example}

The above examples are examples of $G$-structures associated to various geometric structures. 
All of them obtained as reductions of the coframe bundle $B(M)$. 
\begin{exercise}
Describe the $G$-structure corresponding to a nonvanishing covector field $\xi$ on the plane $\mathbb{R}^2$.
\begin{solution}
The total space $P$ consists of coframes $\left\{ \xi(x),e^2 \right\}_x$, where $e^2$ is arbitrary but linearly independent with $\xi(x)$. The group $G$ consists of matrices 
$\left[
\begin{array}{cc}
1 & a
\\
0 & b
\end{array}
\right]$, where $b \ne 0$. 
\end{solution}
\end{exercise}

\section*{Summary of Lecture 2}
The main problem is ``How to construct invariants of a geometric structure''?
\begin{itemize}
\item
The invariants of a geometric structure can be constructed via coefficients of the derivation equations of an adapted frame.
\item
Adapted frames form a $G$-principal subbundle of the coframe bundle $B(M)$, a $G$-structure. 
The group $G$ is the group of linear transformations sending adapted frames to adapted frames. 
\item
The main goal of the Cartan reduction is to reduce a given $G$-structure to a minimal possible $G_1$.
This means that we make the class of adapted frames, and the group $G$, as small as possible.
In particular, if $G_1 = \left\{ e \right\}$, then we get the unique frame field associated to the given geometric structure. 
So the coefficients of the corresponding derivation equations are invariants.
 
A section of a bundle associated with the coframe bundle which takes values in one orbit of $G$-action, determines a $G$-structure. This is the zero step of the Cartan reduction.
\end{itemize}

\textbf{What will we do in the next lecture?} In the next lecture we will explain how to make the other steps of the Cartan reduction.

\Closesolutionfile{ans}
\section*{Answers to exercises}
\input{answers_lecture_2}


\begin{thebibliography}{99}

\bibitem{KN} S.\,Kobayashi, K.\,Nomizu, \emph{Foundations of Differential Geometry}, Vols. I and II, Interscience, London, 1963. 

\bibitem{Montgomery} \emph{R.\,Montgomery}, A tour of Subriemannian Geometries, Their Geodesics and Applications,  Mathematical Surveys and Monographs, V. 91, AMS, Providence, 2002. 
\end{thebibliography}
\end{document}